\documentclass[runningheads]{llncs}
\usepackage[T1]{fontenc}
\usepackage{graphicx}
\usepackage{color} 

\usepackage{amsmath, amssymb, xparse}%
\DeclareMathOperator{\Tr}{Tr}

\usepackage{enumitem}
\setlist[enumerate]{leftmargin=.5in}
\setlist[itemize]{leftmargin=.5in}

\usepackage[
  bookmarks,
  linkcolor= blue,
  citecolor= blue,
  filecolor= blue,
  urlcolor=  blue,
  colorlinks=true,
  hyperindex=true,
  pdftitle={Allelopathy and competition},
  pdfauthor={T. Hmidhi and R. Fekih-Salem}
]{hyperref}
\graphicspath{{Images/}}
\begin{document}
\title{AM2 model with a series configuration of interconnected chemostats and distinct removal rates}

\titlerunning{AM2 Interconnected Chemostats in Series}

%
\author{Thamer Hmidhi\inst{1}
\and Radhouane Fekih-Salem\inst{1}\orcidID{0000-0003-1168-4930} 
}
\authorrunning{T. Hmidhi and R. Fekih-Salem}
%
\institute{University of Tunis El Manar, National Engineering School of Tunis, LAMSIN, 1002, Tunis, Tunisia \\
\email{thamer.hmidhi@enit.utm.tn}\\
\email{radhouane.fekih-salem@enit.utm.tn}
}
\maketitle              
\begin{abstract}
We investigate the dynamics of the AM2 model in a serial configuration of two interconnected chemostats with distinct dilution rates. The system is described by nonlinear differential equations, for which the usual reduction from an eight-dimensional system to a four-dimensional one is no longer possible due to the distinct dilution rates. We provide a complete qualitative study by characterizing all steady states and establishing necessary and sufficient conditions for the existence, multiplicity, and local stability of nine types of equilibria, expressed in terms of key operating parameters. We show that coexistence equilibria involving two species in the second bioreactor are not always stable, even when they exist, and require an additional condition related to the trace of the Jacobian matrix. Identifying parameter sets leading to Hopf bifurcations and limit cycles remains an open problem. These results provide new insights into the dynamics of interconnected bioreactors and motivate further investigations on oscillatory behaviors.

\end{abstract}
\section{Introduction}                                 \label{SecIntro}
Anaerobic Digestion (AD) is a complex biological process in which organic matter is converted into biogas, mainly methane, in the absence of oxygen. Due to its strong nonlinearity and the multiplicity of interacting mechanisms, detailed models such as ADM1 remain difficult to analyze. This has motivated the development of reduced models, among which the two-step AM2 model captures the essential dynamics while remaining analytically tractable.

In parallel, spatial heterogeneity in bioreactors has led to the study of structured configurations such as serially interconnected chemostats. These configurations provide a more realistic representation of industrial processes and may enhance performance in terms of substrate conversion and biogas production, while introducing richer dynamical behaviors compared to a single reactor.

Within the classical chemostat framework, Haider et al.~\cite{HaidarMBE2011} showed that an input substrate threshold can reverse the advantages between series and parallel configurations. Dali Youcef et al.~\cite{DaliYoucefMBE2020,DaliYoucefBMB2022} further analyzed the performance of two chemostats in series, filling a gap in the literature. In parallel, Benyahia et al.~\cite{BenyahiaJPC2012} investigated the asymptotic behavior of the AM2 model.

More recently, in~\cite{HmidhiBMB2025}, an anaerobic digestion model with two interconnected chemostats in series was analyzed in the particular case of identical dilution rates, leading to a reduced four-dimensional system and a complete equilibrium analysis. However, allowing distinct dilution rates in each reactor provides a more realistic framework, as already emphasized in the AM2 model of Bernard et al.~\cite{BernardBB2001}. 
Motivated by this observation, the present work extends the previous analysis by considering different removal rates in the two reactors. Let $r\in(0,1)$ denote the fraction of the total volume such that $r_1V=rV$ and $r_2V=(1-r)V$. The dynamics are described by the following system of ordinary differential equations.
\begin{equation}                                  \label{ModelAM2S2C}
\left\{\begin{array}{lll}
\dot{S}_1^1  &=& D_1\left(S_1^{\text{in}}- S_1^1\right)-k_1\mu_1\left(S_1^1\right) X_1^1, \\
\dot{X}_1^1  &=&\left(\mu_1\left(S_1^1\right) - \alpha 
 D_1\right)X_1^1, \\
\dot{S}_2^1  &=& D_1\left(S_2^{\text{in}}- S_2^1\right)+k_2\mu_1\left(S_1^1\right)X_1^1-k_3\mu_2\left(S_2^1\right) X_2^1,  \\
\dot{X}_2^1  &=&\left(\mu_2\left(S_2^1\right) - \alpha 
 D_1\right)X_2^1,\\
\dot{S}_1^2  &=& D_2\left(S_1^1- S_1^2\right)-k_1\mu_1\left(S_1^2\right) X_1^2, \\
\dot{X}_1^2  &=& \alpha D_2\left(X_1^1- X_1^2\right)+\mu_1\left(S_1^2\right)X_1^2,  \\
\dot{S}_2^2  &=& D_2\left(S_2^1- S_2^2\right)+k_2\mu_1\left(S_1^2\right)X_1^2-k_3\mu_2\left(S_2^2\right) X_2^2 ,\\
\dot{X}_2^2  &=& \alpha D_2\left(X_2^1- X_2^2\right)+\mu_2\left(S_2^2\right)X_2^2,
\end{array}\right.
\end{equation}
where \(S_i^j\) and \(X_i^j\) (\(i,j = 1,2\)) denote the substrate and biomass concentrations in each bioreactor \(j\); \(S_i^{\mathrm{in}}\) are the input concentrations of the substrates in the first bioreactor with flow rate \(Q\), and \(D = Q/V\) is the dilution rate. The removal rates \(\alpha D_i\) for chemostat \(i\) are given by
\[
\alpha D_i = \frac{\alpha D}{r_i} = \frac{\alpha Q}{r_i V},
\]
where \(\alpha \in (0,1)\) decouples the hydraulic retention time (HRT) and the solid retention time (SRT).
The main objective of this work is to provide a complete mathematical analysis of the AM2 model in a serial configuration of two interconnected chemostats with distinct removal rates. We characterize all steady states and establish necessary and sufficient conditions for their existence, multiplicity, and local stability. The analysis covers nine types of equilibria and highlights the impact of heterogeneous removal rates on the qualitative behavior of the system, particularly on stability and coexistence in the second bioreactor.

The paper is organized as follows. Section \ref{SecAssumAnMath} introduces the assumptions on the growth functions and preliminary results. We then analyze the existence, multiplicity, and local stability of equilibria in terms of the operating parameters \(D\), \(S_1^{\mathrm{in}}\), \(S_2^{\mathrm{in}}\), and \(r\). Section \ref{SecConc} presents concluding remarks on the influence of heterogeneous dilution rates and possible extensions. All technical proofs are given in Appendix \ref{SecAppProofs}.
\section{Assumptions on the model and mathematical analysis}             \label{SecAssumAnMath}
 Consider system \eqref{ModelAM2S2C}. We assume that the growth functions \(\mu_1(S_1)\) and \(\mu_2(S_2)\) belong to \(\mathcal{C}^1(\mathbb{R}_+)\) and satisfy the following hypotheses.
\begin{description}
\item[H1:] $\mu_1(0)=0$, $\mu_1(+\infty)=m_1$, and $\mu_1'(S_1)>0$ for all $S_1>0$.
\item[H2:] $\mu_2(0)=0$, $\mu_2(+\infty)=0$, and there exists $S_2^{\text{m}}>0$ such that $\mu_2 '(S_2)>0$, for all $0<S_2<S_2^{\text{m}}$ and $\mu_2'(S_2)<0$, for all $S_2>S_2^{\text{m}}$.
\end{description}
To preserve the biological meaning of \eqref{ModelAM2S2C}, we prove the following result.
\begin{proposition}      \label{PropModDef}
All solutions of model \eqref{ModelAM2S2C} with nonnegative initial conditions remain nonnegative and bounded for all positive times. Let \(\xi = \bigl(S_1^1, X_1^1, S_2^1, X_2^1,\\ S_1^2, X_1^2, S_2^2, X_2^2\bigr)\) be the state vector. The set
\[
\Omega = \left\{ \xi \in \mathbb{R}_+^8 : S_1^i + (k_1 - k_2) X_1^i + S_2^i + k_3 X_2^i \leq \frac{S_1^{\mathrm{in}} + S_2^{\mathrm{in}}}{\alpha^{i}},\; i = 1,2 \right\}
\]
is positively invariant and is a global attractor for \eqref{ModelAM2S2C}.
\end{proposition}
 
Next, the existence and asymptotic stability of steady states in the eight-dimensional system \eqref{ModelAM2S2C} follow from the properties of its constituent subsystems. The proof of this result is given in Appendix~\ref{SecAppProofs}.
\begin{proposition}                              \label{PropExist}
Assume that Hypotheses \textbf{H1} and \textbf{H2} hold. Then, the nine types of steady states of (\ref{ModelAM2S2C}) and their components are given in Table \ref{TabComSS}. Their existence conditions are given in Table \ref{TableCondExisStab}.
\end{proposition}
\begin{table}[ht]
\caption{Auxiliary functions where $D_i:=D/r_i$, $D_i^{rm}:=r_i\mu_2\left(S_2^{\text{m}}\right)$ and $r_0=1$, for $i=0,1,2$ and $X_1^{2*}$ is the unique solution of equation $f_1(x)=g_1(x)$ with the functions $f_1$ and $g_1$ are defined in Table \ref{TabFunc2}.}  \label{TabFunc1}
\vspace{-8mm}
\begin{center}
\begin{tabular}{ @{\hspace{1mm}}l@{\hspace{1mm}}  @{\hspace{1mm}}l@{\hspace{1mm}} }	
\hline
                                         &  Definition
\\\hline
$\lambda_1^i(D,r)$                       &
\begin{tabular}{l}
$\lambda_1^i(D,r)$ is the unique solution of equation $\mu_1(S_1) =\alpha D_i$. \\
It is defined for $0 < D < r_im_1/\alpha$. \\
If $D \geq r_im_1/\alpha$, by convention we let $\lambda_1^i(D,r) = +\infty$.
\end{tabular}
\\ \hline
$\lambda_2^{ij}(D,r)$                     &
\begin{tabular}{l}
$\lambda_2^{i1}(D,r)<\lambda_2^{i2}(D,r)$ are the two solutions of equation $\mu_2\left(S_2\right) = \alpha D_i$. \\
They are defined for $0 < D \leq D_i^{rm}/\alpha$. \\
If $D > D_i^{rm}/\alpha$, by convention we let $\lambda_2^{ij}(D,r) = +\infty$.
\end{tabular}
\\ \hline
$F_{1j}\left(D,r,S_2^{\text{in}}\right)$         &
\begin{tabular}{l}
$F_{1j}\left(D,r,S_2^{\text{in}}\right) =\lambda_1^1(D,r)+\frac{k_1}{k_2}\left(\lambda_2^{1j}(D,r)-S_2^{\text{in}}\right)$. \\
It is defined for $0 < D < \min\left(r_1m_1, D_1^{\text{m}}\right)$.
\end{tabular}
\\ \hline
$F_{2j}\left(D,r,S_2^{\text{in}}\right)$         &
\begin{tabular}{l}
$F_{2j}\left(D,r,S_2^{\text{in}}\right) =\lambda_1^2(D,r)+\frac{k_1}{k_2}\left(\lambda_2^{2j}(D,r)-S_2^{\text{in}}\right)$. \\
It is defined for $0 < D < \min\left(r_2m_1, D_2^{\text{m}}\right)$.
\end{tabular}
\\ \hline
$\phi_j\left(D,r,S_1^{\text{in}},S_2^{\text{in}}\right)$         &
\begin{tabular}{l} $\phi_j\left(D,r,S_1^{\text{in}},S_2^{\text{in}}\right)=S_2^{\text{in}}+\alpha k_2X_1^{2*}\left(D,r,S_1^{\text{in}}\right)-\lambda_2^{2j}(D,r)$.\\
It is defined for $0 < D \leq D_2^{\text{m}}$ and $S_1^{\text{in}}>\lambda_1^1$.
\end{tabular}
\\ \hline
\end{tabular}
\end{center}
\vspace{-2mm}
\end{table}

\begin{table}[!ht]
\caption{Components of all equilibrium points of the eight-dimensional system (\ref{ModelAM2S2C}). Functions $\lambda_1^i$, $\lambda_2^{1i}$, $\lambda_2^{2i}$, $F_{ij}$, and $\phi_i$ are defined in Table \ref{TabFunc1}. Variables with a $^*$ are defined implicitly by equations noted below.}
\label{TabComSS}
\vspace{-4mm}
\begin{center}
\renewcommand{\arraystretch}{1.3}
\begin{tabular}{l c c}
\hline
      & $(S_1^{1*}, X_1^{1*}, S_2^{1*}, X_2^{1*})$ & $(S_1^{2*}, X_1^{2*}, S_2^{2*}, X_2^{2*})$ \\
\hline
$\mathcal{E}_{00}^{00}$ & $(S_1^{\text{in}}, 0, S_2^{\text{in}}, 0)$ & $(S_1^{\text{in}}, 0, S_2^{\text{in}}, 0)$ \\
$\mathcal{E}_{00}^{0i}$ & $(S_1^{\text{in}}, 0, S_2^{\text{in}}, 0)$ & $\left(S_1^{\text{in}}, 0, \lambda_2^{2i}, \dfrac{S_2^{\text{in}} - \lambda_2^{2i}}{\alpha k_3}\right)$ \\
$\mathcal{E}_{00}^{10}$ & $(S_1^{\text{in}}, 0, S_2^{\text{in}}, 0)$ & $\left(\lambda_1^2, \dfrac{S_1^{\text{in}} - \lambda_1^2}{\alpha k_1}, S_2^{2*}, 0\right)^{\dagger}$ \\
$\mathcal{E}_{00}^{1i}$ & $(S_1^{\text{in}}, 0, S_2^{\text{in}}, 0)$ & $\left(\lambda_1^2, \dfrac{S_1^{\text{in}} - \lambda_1^2}{\alpha k_1}, \lambda_2^{2i}, \dfrac{k_2(S_1^{\text{in}} - F_{2i})}{\alpha k_1 k_3}\right)$ \\
$\mathcal{E}_{10}^{10}$ & $\left(\lambda_1^1, \dfrac{S_1^{\text{in}} - \lambda_1^1}{\alpha k_1}, S_2^{1*}, 0\right)$ & $\left(S_1^{2*}, X_1^{2*}, S_2^{2*}, 0\right)^{\ddagger}$ \\
$\mathcal{E}_{10}^{1i}$ & $\left(\lambda_1^1, \dfrac{S_1^{\text{in}} - \lambda_1^1}{\alpha k_1}, S_2^{1*}, 0\right)$ & $\left(S_1^{2*}, X_1^{2*}, \lambda_2^{2i}, \dfrac{\phi_i}{\alpha k_3}\right)^{\ddagger}$ \\
$\mathcal{E}_{0i}^{01}$ & $\left(S_1^{\text{in}}, 0, \lambda_2^{1i}, \dfrac{S_2^{\text{in}} - \lambda_2^{1i}}{\alpha k_3}\right)$ & $\left(S_1^{\text{in}}, 0, S_2^{2*}, X_2^{2*}\right)^{\star}$ \\
$\mathcal{E}_{0i}^{11}$ & $\left(S_1^{\text{in}}, 0, \lambda_2^{1i}, \dfrac{S_2^{\text{in}} - \lambda_2^{1i}}{\alpha k_3}\right)$ & $\left(\lambda_1^2, \dfrac{S_1^{\text{in}} - \lambda_1^2}{\alpha k_1}, S_{21}^{2*}, X_2^{2*}\right)^{\star}$ \\
$\mathcal{E}_{1i}^{11}$ & $\left(\lambda_1^1, \dfrac{S_1^{\text{in}} - \lambda_1^1}{\alpha k_1}, \lambda_2^{1i}, \dfrac{k_2(S_1^{\text{in}} - F_{1i})}{\alpha k_1 k_3} \right)$ & $\left(S_1^{2*}, X_1^{2*}, S_{22}^{2*}, X_2^{2*}\right)^{\ddagger,\star}$ \\
\hline
\end{tabular}
\end{center}
\vspace{-2mm}
\begin{minipage}{0.95\textwidth}
\footnotesize
\textbf{Notes:}
$^{\dagger}$ $S_2^{2*}=S_2^{\mathrm{in}}+ \frac{k_2}{k_1}\bigl(S_1^{\mathrm{in}}-\lambda_1^2\bigr)$. 
$^{\ddagger}$ $S_2^{1*}=S_2^{\mathrm{in}}+ \frac{k_2}{k_1}\bigl(S_1^{\mathrm{in}}-\lambda_1^1\bigr)$; $X_1^{2*}$ is the unique solution of $f_1(x) = g_1(x)$; $S_1^{2*} = S_1^{\mathrm{in}} - \alpha k_1 X_1^{2*}$; $S_2^{2*} = S_2^{\mathrm{in}} + \alpha k_2 X_1^{2*}$. 
$^{\star}$ $X_2^{2*}$ is a solution of $f_2(x) = g_2(x)$; $S_2^{2*} = S_2^{\mathrm{in}} - \alpha k_3 X_2^{2*}$; $S_{21}^{2*}=\frac{k_2}{k_1}\bigl(S_1^{\mathrm{in}}-\lambda_1^2\bigr)+S_2^{\mathrm{in}}-\alpha k_3X_2^{2*}$; $S_{22}^{2*}=S_2^{\mathrm{in}}+\alpha k_2X_1^{2*}-\alpha k_3 X_2^{2*}$.
\end{minipage}
\end{table}
{\small
\begin{table}[!ht]
\begin{center}
\caption{Existence and stability conditions of steady states of system~\eqref{ModelAM2S2C} are summarized, except for \(\mathcal{E}_{00}^{02}\), \(\mathcal{E}_{00}^{12}\), \(\mathcal{E}_{10}^{12}\), \(\mathcal{E}_{02}^{01}\), \(\mathcal{E}_{02}^{11}\) and \(\mathcal{E}_{12}^{11}\), which are unstable. The functions \(\lambda_1^i\), \(\lambda_2^{1i}\), \(\lambda_2^{2i}\), \(F_{ij}\) and \(\phi_i\) (\(i,j=1,2\)) are defined in Table~\ref{TabFunc1}. The functions \(f_1\), \(f_2\), \(g_1\) and \(g_2\) are defined in Table~\ref{TabFunc2}, while \(X_1^{2*}\) is the unique solution of \(f_1(x)=g_1(x)\) and \(X_2^{2*}\) is a solution of \(f_2(x)=g_2(x)\). \(\Tr\bigl(J_3^3(\mathcal{E}_{01}^{11})\bigr)\) [resp. \(\Tr\bigl(J_3^3(\mathcal{E}_{11}^{11})\bigr)\)] is the trace of \(J_3^3(\mathcal{E}_{01}^{11})\) [resp. \(J_3^3(\mathcal{E}_{11}^{11})\)], where \(J_3^3(\mathcal{E}_{01}^{11})\) [resp. \(J_3^3(\mathcal{E}_{11}^{11})\)] is the Jacobian matrix of system~\eqref{ModelAM2S2C} at the equilibrium \(\mathcal{E}_{01}^{11}\) [resp. \(\mathcal{E}_{11}^{11}\)] and is given in \eqref{J33E01-11} [resp. \eqref{J33E11-11}].}\label{TableCondExisStab}
\vspace{-0.2mm}
\begin{tabular}{ @{\hspace{1mm}}l@{\hspace{1mm}} @{\hspace{1mm}}l@{\hspace{1mm}} @{\hspace{1mm}}l@{\hspace{1mm}} }
\hline
                         & Existence conditions                           &   Stability conditions with $i=1$ \\\hline 
$\mathcal{E}_{00}^{00}$  & Always exists         &        $S_1^{\text{in}}<\min\left(\lambda_1^1, \lambda_1^2\right)$, $S_2^{\text{in}}\notin\left[\min\left(\lambda_2^{11}, \lambda_2^{21}\right),\max\left(\lambda_2^{12}, \lambda_2^{22}\right)\right]$    \\
 $\mathcal{E}_{00}^{0i}$& $S_2^{\text{in}}>\lambda_2^{2i}$   & $S_1^{\text{in}}<\min\left(\lambda_1^1, \lambda_1^2\right)$, $S_2^{\text{in}}\notin\left[\lambda_2^{11}, \lambda_2^{12}\right]$    \\
$\mathcal{E}_{00}^{10}$  & $S_1^{\text{in}}>\lambda_1^2$  &  $S_1^{\text{in}}<\lambda_1^1$, $S_2^{\text{in}}\notin\left[\lambda_2^{11}, \lambda_2^{12}\right]$, $S_2^{\text{in}}+\frac{k_2\left(S_1^{\text{in}}-\lambda_1^2\right)}{k_1} \notin\left[\lambda_2^{21}, \lambda_2^{22}\right]$  \\
$\mathcal{E}_{00}^{1i}$  & $S_1^{\text{in}}>\max\left(\lambda_1^2,F_{2i}\right)$ & $S_1^{\text{in}}<\lambda_1^1$, $S_2^{\text{in}}\notin\left[\lambda_2^{11}, \lambda_2^{12}\right]$ \\
$\mathcal{E}_{10}^{10}$  & $S_1^{\text{in}}>\lambda_1^1$ &  $S_2^{\text{in}}+\frac{k_2}{k_1}\left(S_1^{\text{in}}-\lambda_1^1\right) \notin\left[\lambda_2^{11}, \lambda_2^{12}\right]$ and $\phi_1<0$ or $\phi_2>0$\\
$\mathcal{E}_{10}^{1i}$  & $S_1^{\text{in}}>\lambda_1^1$, $\phi_i>0$ & $S_2^{\text{in}}+\frac{k_2}{k_1}\left(S_1^{\text{in}}-\lambda_1^1\right) \notin\left[\lambda_2^{11}, \lambda_2^{12}\right]$ \\
$\mathcal{E}_{0i}^{01}$  & $S_2^{\text{in}}>\lambda_2^{1i}$ &  $S_1^{\text{in}}<\min\left(\lambda_1^1, \lambda_1^2\right)$\\
$\mathcal{E}_{0i}^{11}$  & $S_1^{\text{in}}>\lambda_1^2$, $S_2^{\text{in}}>\lambda_2^{1i}$  & $S_1^{\text{in}}<\lambda_1^1$, $g_2'(X_2^{2*})>f_2'(X_2^{2*})$ and $\Tr(J_3^3(\mathcal{E}_{01}^{11}))<0$  
\\
$\mathcal{E}_{1i}^{11}$  & $S_1^{\text{in}}>\max\left(\lambda_1^1,F_{1i}\right)$ & $g_2'(X_2^{2*})>f_2'(X_2^{2*})$ and $\Tr(J_3^3(\mathcal{E}_{11}^{11}))<0$.
\\ \hline
\end{tabular}
\end{center}
\vspace{-8mm}
\end{table}
}
The following proposition establishes the existence and multiplicity of equilibria in \eqref{ModelAM2S2C}. The proof is analogous to that of Proposition~3 in \cite{HmidhiBMB2025} by substituting the equilibrium point components with those of \eqref{ModelAM2S2C}, and is therefore omitted.
\begin{proposition}                                 \label{propmultip}
Assume that the existence conditions in Table \ref{TableCondExisStab} hold for the equilibria
$\mathcal{E}_{00}^{kl}$, $\mathcal{E}_{10}^{1l}$, $\mathcal{E}_{0i}^{k1}$ and $\mathcal{E}_{1i}^{11}$, $i=1,2$, $k=0,1$, and $l=0,1,2$. 
Let $x_2^{\text{m}}$ be the maximum value of $f_2$ (see Lemma \ref{lemH2}).
\begin{itemize}[leftmargin=*]
\item The equilibria $\mathcal{E}_{00}^{kl}$, $\mathcal{E}_{10}^{1l}$ and $\mathcal{E}_{01}^{01}$ are unique.
\item There exists at least one equilibrium of the form $\mathcal{E}_{02}^{01}$. Generically, the system has an odd number of equilibria of the form $\mathcal{E}_{02}^{01}$.
\item If $X_2^{1*}> x_2^{\text{m}}$, the equilibrium $\mathcal{E}_{0i}^{11}$ and $\mathcal{E}_{1i}^{11}$ are unique. If $X_2^{1*}< x_2^{\text{m}}$, then there exists at least one equilibrium of the form $\mathcal{E}_{0i}^{11}$ and $\mathcal{E}_{1i}^{11}$. Generically, the system has an odd number of equilibria of the form $\mathcal{E}_{0i}^{11}$ and $\mathcal{E}_{1i}^{11}$, $i=1,2$.
\end{itemize}
\end{proposition}
The following proposition establishes the stability conditions for the equilibria of System \eqref{ModelAM2S2C}. A proof of this result is presented in Appendix~\ref{SecAppProofs}.
\begin{proposition}                              \label{PropStab}
Assume that Hypotheses \textbf{H1} and \textbf{H2} hold. The steady states $\mathcal{E}_{00}^{02}$, $\mathcal{E}_{00}^{12}$, $\mathcal{E}_{10}^{12}$, $\mathcal{E}_{02}^{01}$, $\mathcal{E}_{02}^{11}$ and $\mathcal{E}_{12}^{11}$ of system (\ref{ModelAM2S2C}) are unstable when they exist. The stability conditions of all other steady states are given in Table \ref{TableCondExisStab}.
\end{proposition}
The following result establishes the local stability of equilibria of the forms \(\mathcal{E}_{01}^{11}\) and \(\mathcal{E}_{11}^{11}\) in the case where \(X_2^{1*} > x_2^{\mathrm{m}}\).
\begin{proposition}\label{propStabE0111E1111}
Assume that $X_2^{1*} > x_2^{\mathrm{m}}$ and that Hypothesis \textbf{(H2)} is satisfied. 
Let $X_2^{2*,1},\dots,X_2^{2*,k}$ be the solutions of $f_2(x)=g_2(x)$ in $(X_2^{1*},d)$, ordered by
\begin{equation}\label{eq:orderingE0111E1111}
X_2^{1*} < X_2^{2*,1} < \cdots < X_2^{2*,k-1} < x_2^{\mathrm{m}} < X_2^{2*,k} < d.
\end{equation}
Then, for $l=1,\dots,k-1$, $\mathcal{E}_{01}^{11,l}$ [resp. $\mathcal{E}_{11}^{11,l}$] is LES iff 
$\Tr\bigl(J_3^3(\mathcal{E}_{01}^{11})\bigr)<0$ 
[resp. $\Tr\bigl(J_3^3(\mathcal{E}_{11}^{11})\bigr)<0$] when $l$ is odd, and unstable whenever it exists when $l$ is even; moreover, $\mathcal{E}_{01}^{11,k}$ [resp. $\mathcal{E}_{11}^{11,k}$] is LES whenever it exists.
\end{proposition}
\begin{remark}
Consider the equilibrium \(\mathcal{E}_{01}^{11,l}\) for \(l = 1, \dots, k-1\) (before \(x_2^{\mathrm{m}}\)), where \(k\) is odd. If the following inequalities hold:
\[
g_2'\bigl(X_2^{2*,l}\bigr) - \frac{1}{\alpha}f_2'\bigl(X_2^{2*,l}\bigr) < 0
\quad \text{and} \quad
g_2'\bigl(X_2^{2*,l}\bigr) - f_2'\bigl(X_2^{2*,l}\bigr) > 0,
\]
then the stability of this equilibrium remains an open problem. Indeed, under these conditions we have \(\det(J_3^3) > 0\), but the sign of \(\Tr(J_3^3)\) could be either positive or negative, leaving the possibility of a limit cycle.
\end{remark}
\section{Conclusion}                                  \label{SecConc}
In this work, we have studied the AM2 model in a serial configuration of two interconnected chemostats with distinct removal rates. The main contributions of this study can be summarized as follows. First, we established a complete mathematical framework describing the model and highlighted the fact that the usual reduction from an eight-dimensional system to a four-dimensional one is no longer possible due to distinct removal rates, which significantly increases the complexity of the analysis. Second, we provided a complete qualitative characterization of the system by determining all steady states and deriving necessary and sufficient conditions for the existence, multiplicity, and local stability of nine types of equilibria in terms of key operating parameters. Third, we showed that coexistence equilibria in the second bioreactor are not always stable, even when they exist, and that their stability may depend on an additional condition involving the trace of the Jacobian matrix. The possible occurrence of Hopf bifurcations and limit cycles remains an open problem.

Future work will focus on the performance analysis of the process in terms of biogas and biomass productivity, as well as the minimization of the outlet substrate concentration. These issues are particularly relevant in wastewater treatment biotechnology and open interesting perspectives for optimization and control of such bioreactor systems.
\appendix
\section{Proofs}  \label{SecAppProofs}
 \begin{proof}[Proposition \ref{PropModDef}]
The proof of the positivity of solutions follows similar arguments to those developed in \cite{FekihSiads2021}, based on the invariance of the positive orthant and standard uniqueness properties of solutions, and is therefore omitted.
Let $Z_1 = S_1^1 + (k_1 - k_2) X_1^1 + S_2^1 + k_3 X_2^1$. From (\ref{ModelAM2S2C}),
\[
\dot{Z}_1 = D_1( S_1^{\mathrm{in}} + S_2^{\mathrm{in}} - S_1^1 - \alpha(k_1-k_2)X_1^1 - S_2^1 - \alpha k_3 X_2^1) \leq D_1( S_1^{\mathrm{in}} + S_2^{\mathrm{in}} - \alpha Z_1).
\]
By Gronwall's lemma,
\[
Z_1(t) \leq \frac{S_1^{\mathrm{in}} + S_2^{\mathrm{in}}}{\alpha} + \left( Z_1(0) - \frac{S_1^{\mathrm{in}} + S_2^{\mathrm{in}}}{\alpha} \right) e^{-D_1 \alpha t} \leq \max\left\{ Z_1(0), \frac{S_1^{\mathrm{in}} + S_2^{\mathrm{in}}}{\alpha} \right\} =: M_1.
\]

Let $Z_2 = S_1^2 + (k_1 - k_2) X_1^2 + S_2^2 + k_3 X_2^2$. Then
\[
\dot{Z}_2 = D_2( S_1^1 + \alpha(k_1-k_2)X_1^1 + S_2^1 + \alpha k_3 X_2^1 - S_1^2 - \alpha(k_1-k_2)X_1^2 - S_2^2 - \alpha k_3 X_2^2).
\]
Since $S_1^1 + \alpha(k_1-k_2)X_1^1 + S_2^1 + \alpha k_3 X_2^1 \leq Z_1 \leq M_1$, we have $\dot{Z}_2 \leq D_2( M_1 - \alpha Z_2)$. Gronwall's lemma gives
\[
Z_2(t) \leq \frac{M_1}{\alpha} + \left( Z_2(0) - \frac{M_1}{\alpha} \right) e^{-D_2 \alpha t} \leq \max\left\{ Z_2(0), \frac{M_1}{\alpha} \right\}.
\]
Thus, $Z_1$ and $Z_2$ are uniformly bounded for all $t \geq 0$.
\end{proof}
To establish the existence and multiplicity of all steady states of system (\ref{ModelAM2S2C}), we first require the following lemmas.
\begin{lemma}                                            \label{lemH1}
Let $f_1$ and $g_1$ be the functions defined in Table \ref{TabFunc2}, and suppose that $X_1^{1*} < S_1^{\text{in}}/\alpha k_1$. Assume that \textbf{(H1)} holds. Then the equation $f_1(x) = g_1(x)$ has a unique solution $X_1^{2*}$ in the interval $\left(X_1^{1*}, S_1^{\text{in}}/\alpha k_1\right)$.
\end{lemma}
\vspace{-0.5cm}
\begin{table}[!ht]
\caption{Auxiliary functions and their domains of definition, where $d = X_2^{1*}+(S_2^{1*}-\alpha k_2\left(X_1^{1*}-X_1^{2*}\right))/(\alpha k_3)$, while $X_i^{1*}$, $i=1,2$, $S_i^{1*}$, and $X_1^{2*}$ denote the components of the equilibria $\mathcal{E}_{10}^{10}$, $\mathcal{E}_{10}^{1i}$, $\mathcal{E}_{0i}^{01}$, $\mathcal{E}_{0i}^{11}$, and $\mathcal{E}_{1i}^{11}$ defined in Table \ref{TabComSS}.}\label{TabFunc2}
\vspace{-1.8mm}
\begin{center}
\begin{tabular}{llll}
\hline
Auxiliary functions                                    &  Domains of definition \\
\hline
$g_i(x):=\alpha D_2\left(\frac{ x-X_i^{1*}}{x}\right)$           &  $\left(0, +\infty\right)$ \\[1ex]
$f_1(x):=\mu_1\left(S_1^{1*}+\alpha k_1\left(X_1^{1*}- x\right)\right)=\mu_1\left(S_1^{\text{in}}-\alpha k_1x\right)$ & $\left[0, \, X_1^{1*}+\dfrac{S_1^{1*}}{\alpha k_1}\right]$  \\[1ex]
$f_2(x):=\mu_2\!\left(S_2^{1*}-\alpha k_2\left(X_1^{1*}- X_1^{2*}\right)+\alpha k_3\left(X_2^{1*}-x\right)\right)$      & $\left[0,\, d \right]$ \\
\hline
\end{tabular}
\end{center}
\vspace{-8mm}
\end{table}
\begin{proof}[Lemma \ref{lemH1}]
Under \textbf{(H1)}, $f_1$ is nonnegative, decreasing, continuous on $[0, S_1^{\text{in}}/\alpha k_1]$. $g_1$ is nonnegative, increasing hyperbolic, continuous on $[X_1^{1*},+\infty)$ with $g_1'(x)=\alpha D_2 X_1^{1*}/x^2>0$.

Define $F_1(x):=f_1(x)-g_1(x)$ on $[X_1^{1*}, S_1^{\text{in}}/\alpha k_1]$. Then $F_1$ is decreasing from $F_1(X_1^{1*})=\mu_1(S_1^{\text{in}}-\alpha k_1 X_1^{1*})>0$ to $F_1(S_1^{\text{in}}/(\alpha k_1))=-g_1(S_1^{\text{in}}/(\alpha k_1))<0$. Thus, $f_1(x)=g_1(x)$ has a unique solution in $(X_1^{1*}, S_1^{\text{in}}/\alpha k_1)$.
\end{proof}
\begin{lemma}\label{lemH2}
Let $f_2,g_2$ be defined in Table \ref{TabFunc2}, and let $d$ be the unique root of $f_2(x)=0$ (see Table \ref{TabFunc2}), and define
\[
x_2^{\text{m}} := X_2^{1*} + \frac{S_2^{1*}  - \alpha k_2 \left(X_1^{1*} - X_1^{2*}\right) - S_2^{\text{m}}}{\alpha k_3},
\]
the value where $f_2$ reaches its maximum on $[0,d]$. Assume $X_1^{1*} \leq X_1^{2*}$ and \textbf{(H2)}.
\begin{enumerate}[leftmargin=*,label=\raisebox{0.25ex}{\tiny$\bullet$}]
 \item If $X_2^{1*} \geq x_2^{\text{m}}$, $f_2(x)=g_2(x)$ has a unique solution in $(X_2^{1*}, d)$. 
\item If $X_2^{1*} < x_2^{\text{m}}$, $f_2(x)=g_2(x)$ has solutions $X_2^{2*,1},\dots,X_2^{2*,k}$ in $(X_2^{1*}, d)$ with
\[
X_2^{1*} < X_2^{2*,1} < \cdots < X_2^{2*,k-1} < x_2^{\mathrm{m}} < X_2^{2*,k} < d,
\]
where $k$ is the number of solutions. Generically, $k$ is odd.
\end{enumerate}
\end{lemma}
\begin{proof}[Lemma \ref{lemH2}]
Under \textbf{(H2)}, $f_2$ is nonnegative, continuous, and unimodal on $[0,d]$ (increasing on $[0,x_2^{\text{m}}]$, decreasing on $[x_2^{\text{m}},d]$). $g_2$ is nonnegative, continuous on $[X_2^{1*},+\infty)$, and increasing since $g_2'(x)=\alpha D_2 X_2^{1*}/x^2>0$.

Let $F_2(x):=f_2(x)-g_2(x)$. Then $F_2(X_2^{1*})=\mu_2(S_2^{1*}-\alpha k_2(X_1^{1*}-X_1^{2*}))>0$ and $F_2(d)=f_2(d)-g_2(d)=-g_2(d)<0$.

If $X_2^{1*}\geq x_2^{\text{m}}$, $f_2$ is decreasing and $g_2$ increasing on $[X_2^{1*},d]$, so $F_2$ is decreasing, giving a unique solution in $(X_2^{1*},d)$.

If $X_2^{1*}<x_2^{\text{m}}$, $f_2$ increases then decreases while $g_2$ increases. Since $F_2(X_2^{1*})>0$ and $F_2(d)<0$, there is at least one solution. Moreover, under generic conditions, the number of intersections of two continuous functions where one is unimodal and the other is increasing is odd.
\end{proof}
A detailed analysis of the steady states of the upper four-dimensional subsystem, including their existence and local stability conditions, is provided in \cite{BenyahiaJPC2012,SariND2021}. This result will serve as a basis for determining the existence and local stability conditions of the full eight-dimensional system (\ref{ModelAM2S2C}).
\begin{proposition}                      \label{PropExistStabAM2model}
Assume that Hypotheses \textbf{H1} and \textbf{H2} hold. Then, the upper four-dimensional subsystem of (\ref{ModelAM2S2C}) admits exactly six steady states, which are listed in Table \ref{TableAM2}. The conditions for their existence and local stability are summarized in Table \ref{TableStabModelAM2S1S}.
\end{proposition}
\vspace{-0.5cm}
\begin{table}[!ht]
\caption{Components \(S_i^1\) and \(X_i^1\), $i=1,2$ of the steady states of the upper four-dimensional subsystem of (\ref{ModelAM2S2C}). The functions \(\lambda_1^1\), \(\lambda_2^{1i}\), and \(F_{1i}\) are defined in Table\ref{TabFunc1}.}
\label{TableAM2}
\vspace{-0.5cm}
\begin{center}
\begin{tabular}{lcccc}
\hline
           & $S_1^1$       & $S_2^1$     & $X_1^1$ &  $X_2^1$ \\
\hline
$E_{00}$   & $S_1^{\text{in}}$    & $S_2^{\text{in}}$  & $0$     &  $0$     \\
$E_{10}$   & $\lambda_1^1$ & $S_2^{\text{in}} + \frac{k_2}{k_1}(S_1^{\text{in}} - \lambda_1^1)$ & $\frac{S_1^{\text{in}} - \lambda_1^1}{\alpha k_1}$ & $0$ \\
$E_{0i}$   & $S_1^{\text{in}}$    & $\lambda_2^{1i}$  & $0$  & $\frac{S_2^{\text{in}} - \lambda_2^{1i}}{\alpha k_3}$ \\
$E_{1i}$   & $\lambda_1^1$ & $\lambda_2^{1i}$  & $\frac{S_1^{\text{in}} - \lambda_1^1}{\alpha k_1}$ & $\frac{k_2(S_1^{\text{in}} - F_{1i})}{\alpha k_1 k_3}$ \\
\hline
\end{tabular}
\end{center}
\vspace{-1.4cm}
\end{table}
\begin{table}[h!]
\centering
\caption{Existence and stability conditions for the steady states of the upper four-dimensional subsystem of (\ref{ModelAM2S2C}). The functions \(\lambda_1^1\), \(\lambda_2^{1i}\), and \(F_{1i}\), \(i=1,2\), are defined in \ref{TabFunc1}.}
\label{TableStabModelAM2S1S}
\begin{tabular}{lcc}
\hline
          & Existence Conditions  &  Stability Conditions \\
\hline
$E_{00}$  & Always exists & $S_1^{\text{in}} < \lambda_1^1$ and $S_2^{\text{in}} \notin [\lambda_2^{11}, \lambda_2^{12}]$ \\
$E_{10}$  & $S_1^{\text{in}} > \lambda_1^1$ & $S_2^{\text{in}} + \frac{k_2}{k_1}(S_1^{\text{in}} - \lambda_1^1) \notin [\lambda_2^{11}, \lambda_2^{12}]$ \\
$E_{01}$  & $S_2^{\text{in}} > \lambda_2^{11}$ & $S_1^{\text{in}} < \lambda_1^1$ \\
$E_{02}$  & $S_2^{\text{in}} > \lambda_2^{12}$ & Always unstable \\
$E_{11}$  & $S_1^{\text{in}} > \max(\lambda_1^1, F_{11})$ & LES when it exists \\
$E_{12}$  & $S_1^{\text{in}} > \max(\lambda_1^1, F_{12})$ & Always unstable \\
\hline
\end{tabular}
\vspace{-0.8cm}
\end{table}
 \begin{proof}[Proposition \ref{PropExist}]
The first four equations of \eqref{ModelAM2S2C} form the classical AM2 model \cite{BenyahiaJPC2012,BernardBB2001,SariND2021} with volume $\alpha V$. Its steady states are given in Table \ref{TableAM2} with stability in Table \ref{TableStabModelAM2S1S}. The last four equations depend on this subsystem, hence each steady state 
\[
E_{ij}\bigl(S_1^1,X_1^1,S_2^1,X_2^1\bigr),\quad \{i,j\}\in\{0,1,2\}
\]
corresponds to a steady state
\[
\mathcal{E}_{ij}^{kl}\bigl(S_1^1,X_1^1,S_2^1,X_2^1,S_1^2,X_1^2,S_2^2,X_2^2\bigr)
\]
of \eqref{ModelAM2S2C}. The steady states of the lower four-dimensional subsystem are given by the solutions of the following equations:
\begin{align}
D_2\bigl(S_1^1 - S_1^2\bigr) - k_1\mu_1(S_1^2) X_1^2 &= 0, \label{eq1Am2ss2}\\
\alpha D_2\bigl(X_1^1 - X_1^2\bigr) + \mu_1(S_1^2) X_1^2 &= 0, \label{eq2Am2ss2}\\
D_2\bigl(S_2^1 - S_2^2\bigr) + k_2\mu_1(S_1^2) X_1^2 - k_3\mu_2(S_2^2) X_2^2 &= 0, \label{eq3Am2ss1}\\
\alpha D_2\bigl(X_2^1 - X_2^2\bigr) + \mu_2(S_2^2) X_2^2 &= 0. \label{eq4Am2ss1}
\end{align}

\begin{itemize}[leftmargin=*]
\item For \(E_{00}\), we have \(X_1^1 = X_2^1 = 0\). Hence, \eqref{eq1Am2ss2}–\eqref{eq4Am2ss1} become
\begin{align}
D_2\bigl(S_1^{\mathrm{in}} - S_1^2\bigr) - k_1\mu_1(S_1^2) X_1^2 &= 0, \label{eq1SousMod2}\\
\bigl[\mu_1(S_1^2) - \alpha D_2\bigr] X_1^2 &= 0, \label{eq2SousMod2}\\
D_2\bigl(S_2^{\mathrm{in}} - S_2^2\bigr) + k_2\mu_1(S_1^2) X_1^2 - k_3\mu_2(S_2^2) X_2^2 &= 0, \label{eq3SousMod2}\\
\bigl[\mu_2(S_2^2) - \alpha D_2\bigr] X_2^2 &= 0. \label{eq4SousMod2}
\end{align}

\begin{itemize}[leftmargin=*,label=\raisebox{0.25ex}{\tiny$\bullet$}]
\item For \(\mathcal{E}_{00}^{00}\), \(X_1^2 = X_2^2 = 0\). Then \eqref{eq1SousMod2} and \eqref{eq3SousMod2} give \(S_1^2 = S_1^{\mathrm{in}}\) and \(S_2^2 = S_2^{\mathrm{in}}\). This equilibrium always exists.

\item For \(\mathcal{E}_{00}^{10}\), \(X_1^2 > 0\) and \(X_2^2 = 0\). Then \eqref{eq2SousMod2} gives \(\mu_1(S_1^2) = \alpha D_2\). Using the definition of \(\lambda_1^2\) in Table~\ref{TabFunc1}, we obtain \(S_1^2 = \lambda_1^2\). Substituting into \eqref{eq1SousMod2} yields \(X_1^2 = (S_1^{\mathrm{in}} - \lambda_1^2)/(\alpha k_1)\). From \eqref{eq3SousMod2} and \(X_2^2 = 0\), we get \(S_2^2 = S_2^{\mathrm{in}} + k_2(S_1^{\mathrm{in}} - \lambda_1^2)/k_1\). The components of \(\mathcal{E}_{00}^{10}\) are listed in Table~\ref{TabComSS}, and its existence conditions in Table~\ref{TableCondExisStab}.

\item For \(\mathcal{E}_{00}^{0i}\) (\(i = 1,2\)), \(X_1^2 = 0\) and \(X_2^2 > 0\). Using \eqref{eq1SousMod2}, \eqref{eq4SousMod2}, and the definitions of \(\lambda_2^{2i}\) in Table~\ref{TabFunc1}, we obtain \(S_1^2 = S_1^{\mathrm{in}}\) and \(S_2^2 = \lambda_2^{2i}\). Then \eqref{eq3SousMod2} gives \(X_2^2 = (S_2^{\mathrm{in}} - \lambda_2^{2i})/(\alpha k_3)\). The components of \(\mathcal{E}_{00}^{0i}\) are in Table~\ref{TabComSS}, and its existence conditions in Table~\ref{TableCondExisStab}.

\item For \(\mathcal{E}_{00}^{1i}\) (\(i = 1,2\)), \(X_1^2 > 0\) and \(X_2^2 > 0\). Then \eqref{eq2SousMod2} and \eqref{eq4SousMod2} give \(\mu_1(S_1^2) = \alpha D_2\) and \(\mu_2(S_2^2) = \alpha D_2\). From \textbf{(H1)}, \textbf{(H2)}, and Table~\ref{TabFunc1}, we obtain \(S_1^2 = \lambda_1^2\) and \(S_2^2 = \lambda_2^{2i}\). Substituting into \eqref{eq1SousMod2} yields \(X_1^2 = (S_1^{\mathrm{in}} - \lambda_1^2)/(\alpha k_1)\), and into \eqref{eq3SousMod2} yields \(X_2^2 = k_2(S_1^{\mathrm{in}} - F_{2i})/(\alpha k_1 k_3)\).
\end{itemize}
\end{itemize}
\item For \(E_{10}\), we have \(X_1^1 > 0\) and \(X_2^1 = 0\). Hence, \eqref{eq1Am2ss2}–\eqref{eq4Am2ss1} become
\begin{align}
D_2\bigl(S_1^{1*} - S_1^2\bigr) - k_1\mu_1(S_1^2) X_1^2 &= 0, \label{eq1SousMod3}\\
\alpha D_2\bigl(X_1^{1*} - X_1^2\bigr) + \mu_1(S_1^2) X_1^2 &= 0, \label{eq2SousMod3}\\
D_2\bigl(S_2^{1*} - S_2^2\bigr) + k_2\mu_1(S_1^2) X_1^2 - k_3\mu_2(S_2^2) X_2^2 &= 0, \label{eq3SousMod3}\\
\bigl[\mu_2(S_2^2) - \alpha D_2\bigr] X_2^2 &= 0. \label{eq4SousMod3}
\end{align}

\begin{itemize}[leftmargin=*,label=\raisebox{0.25ex}{\tiny$\bullet$}]

\item For \(\mathcal{E}_{10}^{00}\), \(X_1^2 = X_2^2 = 0\). Then \eqref{eq2SousMod3} implies \(X_1^1 = 0\), which contradicts the existence condition of \(E_{10}\) (\(X_1^1 > 0\)). Hence, \(\mathcal{E}_{10}^{00}\) does not exist.

\item For \(\mathcal{E}_{10}^{10}\), \(X_1^2 > 0\) and \(X_2^2 = 0\). Using \eqref{eq1SousMod3} + \(k_1\)\eqref{eq2SousMod3}, we obtain
\[
S_1^2 = S_1^1 + \alpha k_1\bigl(X_1^1 - X_1^2\bigr).
\]
From the expressions of \(S_1^1\) and \(X_1^1\) in Table~\ref{TableAM2}, we get
\begin{equation}\label{S12E1010}
S_1^2 = S_1^{\mathrm{in}} - \alpha k_1 X_1^2.
\end{equation}
Thus, \eqref{eq2SousMod3} yields
\begin{equation}\label{eqf1g1}
\mu_1\bigl(S_1^1 + \alpha k_1(X_1^1 - X_1^2)\bigr) = \alpha D_2\left(\frac{X_1^2 - X_1^1}{X_1^2}\right),
\end{equation}
which is equivalent to \(f_1(x) = g_1(x)\) (see Table~\ref{TabFunc2}). From the components \(S_1^1\) and \(X_1^1\) of \(E_{10}\) in Table~\ref{TableAM2}, it follows that \(X_1^1 < S_1^{\mathrm{in}}/(\alpha k_1)\). By Lemma~\ref{lemH1}, the equation \(f_1(x) = g_1(x)\) has a unique solution \(X_1^{2*} \in \bigl(X_1^1, S_1^{\mathrm{in}}/(\alpha k_1)\bigr)\). Hence, \eqref{eq3SousMod3} \(- k_2\)\eqref{eq2SousMod3} gives
\[
S_2^2 = S_2^1 - \alpha k_2\bigl(X_1^1 - X_1^{2*}\bigr).
\]
From the expressions of \(S_2^1\) and \(X_1^1\) in Table~\ref{TableAM2}, we obtain
\begin{equation}\label{S22E1010}
S_2^2 = S_2^{\mathrm{in}} + \alpha k_2 X_1^{2*}.
\end{equation}
Thus, we obtain the components of \(\mathcal{E}_{10}^{10}\) in Table~\ref{TabComSS} and its existence condition in Table~\ref{TableCondExisStab}.

\item For \(\mathcal{E}_{10}^{0i}\) (\(i = 1,2\)), \(X_1^2 = 0\) and \(X_2^2 > 0\). Then \eqref{eq2SousMod3} implies \(X_1^1 = 0\), which contradicts \(X_1^1 > 0\) for \(E_{10}\). Hence, \(\mathcal{E}_{10}^{0i}\) does not exist.

\item For \(\mathcal{E}_{10}^{1i}\) (\(i = 1,2\)), \(X_1^2 > 0\) and \(X_2^2 > 0\). As for \(\mathcal{E}_{10}^{10}\), \(S_1^2\) is given by \eqref{S12E1010} and \(X_1^{2*}\) is the unique solution of \eqref{eqf1g1}. In addition, \eqref{eq4SousMod3} gives \(\mu_2(S_2^2) = \alpha D_2\). From \textbf{(H2)} and the definition of \(\lambda_2^{2i}\) in Table~\ref{TabFunc1}, we obtain \(S_2^2 = \lambda_2^{2i}\). Then \eqref{eq1SousMod3} yields
\[
X_2^2 = \frac{1}{\alpha k_3}\bigl(S_2^1 - S_2^2\bigr) - \frac{k_2}{k_3}\bigl(X_1^1 - X_1^{2*}\bigr).
\]
Using the expressions of \(S_2^2\), \(S_2^1\), and \(X_1^1\) in Table~\ref{TableAM2}, we obtain
\begin{equation}\label{X22E101i}
X_2^2 = \frac{1}{\alpha k_3}\bigl(S_2^{\mathrm{in}} + \alpha k_2 X_1^{2*} - \lambda_2^{2i}\bigr).
\end{equation}
Hence, the components of \(\mathcal{E}_{10}^{1i}\) are given in Table~\ref{TabComSS}, and the existence conditions follow from Table~\ref{TableCondExisStab}.
\end{itemize}
\item For \(E_{0i}\), we have \(X_1^1 = 0\) and \(X_2^{1i} > 0\), \(i = 1,2\). Hence, \eqref{eq1Am2ss2}–\eqref{eq4Am2ss1} become
\begin{align}
D_2\bigl(S_1^{\mathrm{in}} - S_1^2\bigr) - k_1\mu_1(S_1^2) X_1^2 &= 0, \label{eq1Am2ss4}\\
\bigl[\mu_1(S_1^2) - \alpha D_2\bigr] X_1^2 &= 0, \label{eq2Am2ss4}\\
D_2\bigl(S_2^1 - S_2^2\bigr) + k_2\mu_1(S_1^2) X_1^2 - k_3\mu_2(S_2^2) X_2^2 &= 0, \label{eq3Am2ss4}\\
\alpha D_2\bigl(X_2^1 - X_2^2\bigr) + \mu_2(S_2^2) X_2^2 &= 0. \label{eq4Am2ss4}
\end{align}

\begin{itemize}[leftmargin=*,label=\raisebox{0.25ex}{\tiny$\bullet$}]

\item For \(\mathcal{E}_{0i}^{00}\) [resp. \(\mathcal{E}_{0i}^{10}\)], we have \(X_1^2 = X_2^2 = 0\) [resp. \(X_1^2 > 0\) and \(X_2^2 = 0\)]. Then \eqref{eq4Am2ss4} implies \(X_2^{1i} = 0\), which contradicts \(X_2^{1i} > 0\). Hence, \(\mathcal{E}_{0i}^{00}\) and \(\mathcal{E}_{0i}^{10}\) do not exist.

\item For \(\mathcal{E}_{0i}^{01}\), we have \(X_1^2 = 0\) and \(X_2^2 > 0\). Then \eqref{eq1Am2ss4} gives \(S_1^2 = S_1^{\mathrm{in}}\). Hence, \eqref{eq3Am2ss4} \(+ k_3\)\eqref{eq4Am2ss4} yields
\[
S_2^2 = S_2^1 + \alpha k_3\bigl(X_2^1 - X_2^2\bigr).
\]
From the expressions of \(S_2^1\) and \(X_2^1\) in Table~\ref{TableAM2}, we obtain
\[
S_2^2 = S_2^{\mathrm{in}} - \alpha k_3 X_2^2.
\]
From \eqref{eq4Am2ss4} and the above, \(X_2^2\) must satisfy
\begin{equation}\label{Eqf3g2}
\mu_2\bigl(S_2^1 + \alpha k_3(X_2^1 - X_2^2)\bigr) = \alpha D_2\left(\frac{X_2^2 - X_2^{1i}}{X_2^2}\right),
\end{equation}
which is equivalent to \(f_2(X_2^2) = g_2(X_2^2)\), with \(X_1^1 = X_1^2 = 0\). From the expression of \(d\) in Lemma~\ref{lemH2}, we have \(X_2^{1i} < d\) since \(X_1^{1*} < X_1^{2*}\) (see Lemma~\ref{lemH1}). By Lemma~\ref{lemH2}, the equation \(f_2(x) = g_2(x)\) has at least one solution \(X_2^{2*} \in (X_2^{1i}, d)\). Hence, the components of \(\mathcal{E}_{0i}^{01}\) are given in Table~\ref{TabComSS}, and at least one such equilibrium exists.

\item For \(\mathcal{E}_{0i}^{11}\), we have \(X_1^2 > 0\) and \(X_2^2 > 0\). The components \(S_1^2\) and \(X_1^2\) in Table~\ref{TabComSS} follow from \textbf{(H1)} and equations \eqref{eq2Am2ss4} and \eqref{eq1Am2ss4}, respectively. Then \eqref{eq2Am2ss4} and \eqref{eq3Am2ss4} \(+ k_3\)\eqref{eq4Am2ss4} yield
\[
S_2^2 = S_2^1 + \alpha k_2 X_1^2 + \alpha k_3\bigl(X_2^{1i} - X_2^2\bigr).
\]
From the expressions of \(S_2^1\) and \(X_2^1\) in Table~\ref{TableAM2} and \(X_1^2\) in Table~\ref{TabComSS}, we obtain
\[
S_2^2 = \frac{k_2}{k_1}\bigl(S_1^{\mathrm{in}} - \lambda_1^2\bigr) + S_2^{\mathrm{in}} - \alpha k_3 X_2^2.
\]
In addition, \eqref{eq4Am2ss4} together with the expression for \(S_2^2\) implies that \(X_2^2\) must satisfy
\[
\mu_2\bigl(S_2^1 + \alpha k_2 X_1^2 + \alpha k_3(X_2^1 - X_2^2)\bigr) = \alpha D_2\left(\frac{X_2^2 - X_2^{1i}}{X_2^2}\right),
\]
which is equivalent to \(f_2(x) = g_2(x)\) with \(X_1^1 = 0\). Using the expression of \(d\) in Lemma~\ref{lemH2} and the condition \(X_1^{1*} < X_1^{2*}\) (Lemma~\ref{lemH1}), we have \(X_2^{1i} < d\). By Lemma~\ref{lemH2}, the equation \(f_2(x) = g_2(x)\) has at least one solution \(X_2^{2*} \in (X_2^{1i}, d)\). Consequently, we obtain the components of \(\mathcal{E}_{0i}^{11}\) in Table~\ref{TabComSS} and their existence conditions in Table~\ref{TableCondExisStab}.

\end{itemize}
\item For \(E_{1i}\), we have \(X_1^1 > 0\) and \(X_2^{1i} > 0\), \(i = 1,2\).

\begin{itemize}[leftmargin=*,label=\raisebox{0.25ex}{\tiny$\bullet$}]

\item For \(\mathcal{E}_{1i}^{00}\) [resp. \(\mathcal{E}_{1i}^{10}\), and \(\mathcal{E}_{1i}^{01}\)], we have \(X_1^2 = X_2^2 = 0\) [resp. \(X_1^2 > 0\) and \(X_2^2 = 0\); and \(X_1^2 = 0\) and \(X_2^2 > 0\)]. Then \eqref{eq2Am2ss2} or \eqref{eq4Am2ss1} implies \(X_1^1 = 0\) or \(X_2^{1i} = 0\), which is a contradiction. Hence, these three types of steady states do not exist.

\item For \(\mathcal{E}_{1i}^{11}\), we have \(X_1^2 > 0\) and \(X_2^2 > 0\). Then \eqref{eq1Am2ss2} \(+ k_1\)\eqref{eq2Am2ss2} is equivalent to \eqref{S12E1010}, giving the component \(S_1^2\). In addition, \eqref{eq2Am2ss2} is equivalent to \eqref{eqf1g1}, i.e., \(f_1(x) = g_1(x)\). As for \(\mathcal{E}_{10}^{10}\), we obtain the component \(X_1^{2*}\), which is the unique solution of \(f_1(x) = g_1(x)\). Hence, \eqref{eq3Am2ss1} \(+ k_3\)\eqref{eq4Am2ss1} yields
\[
S_2^2 = S_2^1 - \alpha k_2\bigl(X_1^1 - X_1^{2*}\bigr) + \alpha k_3\bigl(X_2^1 - X_2^2\bigr).
\]
From the expressions of \(S_2^1\), \(X_1^1\), and \(X_2^1\) in Table~\ref{TableAM2}, we obtain
\[
S_2^2 = S_2^{\mathrm{in}} + \alpha k_2 X_1^{2*} - \alpha k_3 X_2^2.
\]
Then, from \eqref{eq4Am2ss1} and the expression for \(S_2^2\), the variable \(X_2^2\) must satisfy
\[
\mu_2\bigl(S_2^1 - \alpha k_2(X_1^1 - X_1^{2*}) + \alpha k_3(X_2^1 - X_2^2)\bigr) = \alpha D_2\left(\frac{X_2^2 - X_2^1}{X_2^2}\right),
\]
which is equivalent to \(f_2(X_2^2) = g_2(X_2^2)\). From the expression of \(d\) in Lemma~\ref{lemH2}, we have \(X_2^{1i} < d\) since \(X_1^1 < X_1^{2*}\) (see Lemma~\ref{lemH1}). By Lemma~\ref{lemH2}, the equation \(f_2(x) = g_2(x)\) has at least one solution in \((X_2^{1i}, d)\). Therefore, we obtain the components of \(\mathcal{E}_{1i}^{11}\) in Table~\ref{TabComSS} and their existence conditions in Table~\ref{TableCondExisStab}.
\end{itemize}
\end{proof}
\begin{proof}[Proposition \ref{PropStab}]
At a steady state $(S_1^1, X_1^1, S_2^1, X_2^1, S_1^2, X_1^2, S_2^2, X_2^2)$, the Jacobian $J$ of \eqref{ModelAM2S2C} is given by the following lower block triangular matrix:
\[
J = \begin{bmatrix} J_1 & 0 \\ J_2 & J_3 \end{bmatrix},\quad
J_1 = \begin{bmatrix}
a_{11} & a_{12} & 0 & 0 \\ a_{21} & a_{22} & 0 & 0 \\ * & * & a_{33} & a_{34} \\ 0 & 0 & a_{43} & a_{44}
\end{bmatrix},\quad
J_3 = \begin{bmatrix}
a_{55} & a_{56} & 0 & 0 \\ a_{65} & a_{66} & 0 & 0 \\ * & * & a_{77} & a_{78} \\ 0 & 0 & a_{87} & a_{88}
\end{bmatrix},
\]
where `$*$` denotes coefficients ($a_{31},a_{32}$ in $J_1$, $a_{75},a_{76}$ in $J_3$) that do not affect the stability of the steady state. The relevant entries are
\[
\begin{aligned}
&a_{11}=-D_1-k_1\mu_1'(S_1^1)X_1^1, & a_{12}=-k_1\mu_1(S_1^1), &  a_{21}=\mu_1'(S_1^1)X_1^1, & a_{22}=\mu_1(S_1^1)-\alpha D_1,\\
&a_{33}=-D_1-k_3\mu_2'(S_2^1)X_2^1, & a_{34}=-k_3\mu_2(S_2^1), & a_{43}=\mu_2'(S_2^1)X_2^1, & a_{44}=\mu_2(S_2^1)-\alpha D_1,\\
&a_{55}=-D_2-k_1\mu_1'(S_1^2)X_1^2, & a_{56}=-k_1\mu_1(S_1^2), & a_{65}=\mu_1'(S_1^2)X_1^2, & a_{66}=\mu_1(S_1^2)-\alpha D_2,\\
&a_{77}=-D_2-k_3\mu_2'(S_2^2)X_2^2, & a_{78}=-k_3\mu_2(S_2^2), & a_{87}=\mu_2'(S_2^2)X_2^2, & a_{88}=\mu_2(S_2^2)-\alpha D_2.
\end{aligned}
\]
$J_1$ and $J_3$ are the Jacobians of the upper and lower four-dimensional subsystems; stability depends on their eigenvalues. Block-diagonalizing
\[
J_3 = \begin{bmatrix} J_3^1 & 0 \\ J_3^2 & J_3^3 \end{bmatrix},\quad
J_3^1 = \begin{bmatrix} a_{55} & a_{56} \\ a_{65} & a_{66} \end{bmatrix},\quad
J_3^3 = \begin{bmatrix} a_{77} & a_{78} \\ a_{87} & a_{88} \end{bmatrix},
\]
gives eigenvalues of $J_3$ as those of $J_3^1$ and $J_3^3$. 
At $E_{00}^{00}=(S_1^{\mathrm{in}},0,S_2^{\mathrm{in}},0)$,
\[
J_3^1 = \begin{bmatrix}-D_2 & -k_1\mu_1(S_1^{\mathrm{in}})\\0 & \mu_1(S_1^{\mathrm{in}})-\alpha D_2\end{bmatrix},\quad
J_3^3 = \begin{bmatrix}-D_2 & -k_3\mu_2(S_2^{\mathrm{in}})\\0 & \mu_2(S_2^{\mathrm{in}})-\alpha D_2\end{bmatrix},
\]
with eigenvalues
$\lambda_1=-D_2$, $\lambda_2=\mu_1(S_1^{\mathrm{in}})-\alpha D_2$, $\lambda_3=-D_2$, $\lambda_4=\mu_2(S_2^{\mathrm{in}})-\alpha D_2$,
negative iff $S_1^{\mathrm{in}}<\lambda_1^2$ and $S_2^{\mathrm{in}}\notin[\lambda_2^{21},\lambda_2^{22}]$. By Table~\ref{TableCondExisStab}, $\mathcal{E}_{00}^{00}$ is stable iff the two conditions therein hold.
 
At \(E_{00}^{0i} = \bigl(S_1^{\mathrm{in}}, 0, \lambda_2^{2i}, X_2^{2i*}\bigr)\), \(i = 1,2\), where the component \(X_2^{2i*}\) is defined in Table~\ref{TabComSS}, the Jacobian matrices \(J_3^1\) and \(J_3^3\) are given by
\begin{equation*}
J_3^1 =
\begin{bmatrix}
-D_2 & -k_1\mu_1(S_1^{\mathrm{in}}) \\[2pt]
0 & \mu_1(S_1^{\mathrm{in}}) - \alpha D_2
\end{bmatrix},\qquad
J_3^3 =
\begin{bmatrix}
-D_2 - k_3\mu_2'(\lambda_2^{2i}) X_2^{2i*} & -\alpha k_3 D_2 \\[2pt]
\mu_2'(\lambda_2^{2i}) X_2^{2i*} & 0
\end{bmatrix}.
\end{equation*}

Similarly to \(E_{00}^{00}\), the eigenvalues of \(J_3^1\) are negative if and only if \(S_1^{\mathrm{in}} < \lambda_1^2\). In addition, the determinant and trace of \(J_3^3\) are given by
\[
\det(J_3^3) = \alpha k_3 D_2 \mu_2'(\lambda_2^{2i}) X_2^{2i*},\qquad
\Tr(J_3^3) = -D_2 - k_3 \mu_2'(\lambda_2^{2i}) X_2^{2i*}.
\]

From \textbf{(H2)}, \(E_{00}^{02}\) is unstable when it exists, since \(\det(J_3^3) < 0\). However, \(E_{00}^{01}\) is LES if and only if \(S_1^{\mathrm{in}} < \lambda_1^2\), because \(\det(J_3^3) > 0\) and \(\Tr(J_3^3) < 0\).

At \(E_{00}^{10} = \bigl(\lambda_1^2, X_1^{2*}, S_2^{2*}, 0\bigr)\), where
\(X_1^{2*} = (S_1^{\mathrm{in}} - \lambda_1^2)/(\alpha k_1)\) and
\(S_2^{2*} = S_2^{\mathrm{in}} + \frac{k_2}{k_1}(S_1^{\mathrm{in}} - \lambda_1^2)\),
the Jacobian matrices \(J_3^1\) and \(J_3^3\) are given by
\begin{equation*}
J_3^1 =
\begin{bmatrix}
-D_2 - k_1\mu_1'(\lambda_1^2) X_1^{2*} & -\alpha k_1 D_2 \\[2pt]
\mu_1'(\lambda_1^2) X_1^{2*} & 0
\end{bmatrix},\qquad
J_3^3 =
\begin{bmatrix}
-D_2 & -k_3\mu_2(S_2^{2*}) \\[2pt]
0 & \mu_2(S_2^{2*}) - \alpha D_2
\end{bmatrix}.
\end{equation*}

Therefore, the determinant and trace of \(J_3^1\) are given by
\[
\det(J_3^1) = \alpha k_1 D_2 \mu_1'(\lambda_1^2) X_1^{2*},\qquad
\Tr(J_3^1) = -D_2 - k_1 \mu_1'(\lambda_1^2) X_1^{2*}.
\]

From \textbf{(H1)}, \(\det(J_3^1) > 0\) and \(\Tr(J_3^1) < 0\). In addition, the eigenvalues of \(J_3^3\) are negative if and only if
\[
\mu_2\!\left(S_2^{\mathrm{in}} + \frac{k_2}{k_1}(S_1^{\mathrm{in}} - \lambda_1^2)\right) < \alpha D_2.
\]
From \textbf{(H2)} and the stability condition of \(E_{00}\) in Table~\ref{TableStabModelAM2S1S}, it follows that \(\mathcal{E}_{00}^{10}\) is LES if and only if the stability condition in Table~\ref{TableCondExisStab} holds.

At \(E_{00}^{1i} = \bigl(\lambda_1^2,\; X_1^{2*},\; \lambda_2^{2i},\; X_2^{2i*}\bigr)\), where the components $ X_1^{2*}$ and \(X_2^{2i*}\) are defined in Table~\ref{TabComSS}, the Jacobian matrices \(J_3^1\) and \(J_3^3\) are given by
\begin{equation*}
J_3^1 =
\begin{bmatrix}
-D_2 - k_1\mu_1'(\lambda_1^2) X_1^{2*} & -\alpha k_1 D_2 \\[4pt]
\mu_1'(\lambda_1^2) X_1^{2*} & 0
\end{bmatrix},\quad
J_3^3 =
\begin{bmatrix}
-D_2 - k_3\mu_2'(\lambda_2^{2i}) X_2^{2i*} & -\alpha k_3 D_2 \\[4pt]
\mu_2'(\lambda_2^{2i}) X_2^{2i*} & 0
\end{bmatrix}.
\end{equation*}

Similarly to \(E_{00}^{10}\), the eigenvalues of \(J_3^1\) are negative. In addition, the determinant and trace of \(J_3^3\) are given by
\[
\det(J_3^3) = \alpha k_3 D_2 \,\mu_2'(\lambda_2^{2i})\, X_2^{2i*},\qquad
\Tr(J_3^3) = -D_2 - k_3\mu_2'(\lambda_2^{2i})\, X_2^{2i*}.
\]

From \textbf{(H2)}, the steady state \(\mathcal{E}_{00}^{12}\) is unstable when it exists, since \(\det(J_3^3) < 0\). However, \(E_{00}^{11}\) is LES because \(\det(J_3^3) > 0\) and \(\Tr(J_3^3) < 0\).

At \(E_{10}^{10} = \bigl(S_1^{2*}, X_1^{2*}, S_2^{2*}, 0\bigr)\), where
\(S_1^{2*} = S_1^{\mathrm{in}} - \alpha k_1 X_1^{2*}\), \(X_1^{2*}\) is the unique solution of \(f_1(X_1^{2*}) = g_1(X_1^{2*})\), and \(S_2^{2*} = S_2^{\mathrm{in}} + \alpha k_2 X_1^{2*}\),
the Jacobian matrices \(J_3^1\) and \(J_3^3\) are given by
\begin{equation*}
J_3^1 =
\begin{bmatrix}
-D_2 - k_1\mu_1'(S_1^{2*}) X_1^{2*} & -k_1\mu_1(S_1^{2*}) \\[4pt]
\mu_1'(S_1^{2*}) X_1^{2*} & \mu_1(S_1^{2*}) - \alpha D_2
\end{bmatrix},\qquad
J_3^3 =
\begin{bmatrix}
-D_2 & -k_3\mu_2(S_2^{2*}) \\[4pt]
0 & \mu_2(S_2^{2*}) - \alpha D_2
\end{bmatrix}.
\end{equation*}

Therefore, the determinant and trace of \(J_3^1\) are given by
\[\begin{aligned}
\det(J_3^1) &= -\bigl(D_2 + k_1\mu_1'(S_1^{2*}) X_1^{2*}\bigr)\bigl(\mu_1(S_1^{2*}) - \alpha D_2\bigr) + k_1\mu_1'(S_1^{2*})\mu_1(S_1^{2*}) X_1^{2*},\\
\Tr(J_3^1) &= -D_2 - k_1\mu_1'(S_1^{2*}) X_1^{2*} + \mu_1(S_1^{2*}) - \alpha D_2.
\end{aligned}
\]
 Using the definitions of \(f_1\) and \(g_1\) in Table~\ref{TabFunc2}, we easily obtain
 \[
{\small \det(J_3^1) = D_2 X_1^{2*}\Bigl[g_1'(X_1^{2*}) - f_1'(X_1^{2*})\Bigr], ~
\Tr(J_3^1) = -D_2 - X_1^{2*}\Bigl[g_1'(X_1^{2*}) - \frac{f_1'(X_1^{2*})}{\alpha}\Bigr].}
\]

Recall that the functions \(x \mapsto g_1(x) - f_1(x)\) and \(x \mapsto g_1(x) - \frac{1}{\alpha}f_1(x)\) are increasing (see the proof of Lemma~\ref{lemH1}). Hence, \(\det(J_3^1) > 0\) and \(\Tr(J_3^1) < 0\).

In addition, the eigenvalues of \(J_3^3\) are negative if and only if \(\mu_2\bigl(S_2^{\mathrm{in}} + \alpha k_2 X_1^{2*}\bigr) < \alpha D_2\). Using \textbf{(H2)} and the stability condition of \(E_{10}\) in Table~\ref{TableStabModelAM2S1S}, \(\mathcal{E}_{10}^{10}\) is LES if and only if the two stability conditions in Table~\ref{TableCondExisStab} hold.

At \(E_{10}^{1i} = \bigl(S_1^{2*}, X_1^{2*}, \lambda_2^{2i}, X_2^{2i*}\bigr)\), \(i = 1,2\), where the components $S_1^{2*}$, $ X_1^{2*}$ and \(X_2^{2i*}\) are defined in Table~\ref{TabComSS}, the Jacobian matrix \(J_3^1\) is the same as at the steady state \(E_{10}^{10}\), and \(J_3^3\) is given by
\begin{equation*}
J_3^3 =
\begin{bmatrix}
-D_2 - k_3\mu_2'(\lambda_2^{2i})\,X_2^{2i*} & -\alpha k_3 D_2 \\[4pt]
\mu_2'(\lambda_2^{2i})\,X_2^{2i*} & 0
\end{bmatrix}.
\end{equation*}

Similarly to \(E_{10}^{10}\), the eigenvalues of \(J_3^1\) are negative. In addition, the determinant and trace of \(J_3^3\) are given by
\[
\det(J_3^3) = \alpha k_3 D_2 \,\mu_2'(\lambda_2^{2i})\, X_2^{2i*},\qquad
\Tr(J_3^3) = -D_2 - k_3 \mu_2'(\lambda_2^{2i})\, X_2^{2i*}.
\]

From \textbf{(H2)}, \(E_{10}^{12}\) is unstable when it exists, since \(\det(J_3^3) < 0\), while \(E_{10}^{11}\) is LES when it exists, because \(\det(J_3^3) > 0\) and \(\Tr(J_3^3) < 0\).

For \(E_{01}^{01} = \bigl(S_1^{\mathrm{in}}, 0, S_2^{2*}, X_2^{2*}\bigr)\), where
\(S_2^{2*} = S_2^{\mathrm{in}} - \alpha k_3 X_2^{2*}\) and \(X_2^{2*}\) is the unique solution of \(f_2(x) = g_2(x)\), the Jacobian matrix \(J_3^1\) is the same as at the steady state \(E_{00}^{00}\), and \(J_3^3\) is given by
\begin{equation*}
J_3^3 =
\begin{bmatrix}
-D_2 - k_3\mu_2'(S_2^{2*})\,X_2^{2*} & -k_3\mu_2(S_2^{2*}) \\[4pt]
\mu_2'(S_2^{2*})\,X_2^{2*} & \mu_2(S_2^{2*}) - \alpha D_2
\end{bmatrix}.
\end{equation*}

Similarly to \(E_{00}^{00}\), the eigenvalues of \(J_3^1\) are negative if and only if \(S_1^{\mathrm{in}} < \lambda_1^2\). In addition, the determinant and trace of \(J_3^3\) are given by
\[
\begin{aligned}
    \det(J_3^3) &= \bigl(D_2 + k_3\mu_2'(S_2^{2*})\,X_2^{2*}\bigr)\bigl(\alpha D_2 - \mu_2(S_2^{2*})\bigr) + k_3 X_2^{2*}\mu_2'(S_2^{2*})\mu_2(S_2^{2*}),\\
    \Tr(J_3^3) &= -D_2 - k_3\mu_2'(S_2^{2*})\,X_2^{2*} + \mu_2(S_2^{2*}) - \alpha D_2.
\end{aligned}
\]
Using the definitions of \(f_2'\) and \(g_2'\) in Lemma~\ref{lemH2}, it follows that
\[
{\small 
\det(J_3^3) = D_2 X_2^{2*}\Bigl[g_2'(X_2^{2*}) - f_2'(X_2^{2*})\Bigr],\\
\Tr(J_3^3) = -D_2 - X_2^{2*}\Bigl[g_2'(X_2^{2*}) - \frac{f_2'(X_2^{2*})}{\alpha}\Bigr].
}\]
Recall that \(X_2^{2*} > x_1^{\mathrm{m}} := \bigl(S_2^{\mathrm{in}} - S_2^{\mathrm{m}}\bigr)/\alpha k_3\), so that the functions \(x \mapsto g_2(x) - f_2(x)\) and \(x \mapsto g_2(x) - \frac{1}{\alpha}f_2(x)\) are increasing (see the proof of Lemma~\ref{lemH2}). Hence, \(\det(J_3^3) > 0\) and \(\Tr(J_3^3) < 0\). Therefore, the stability conditions of \(\mathcal{E}_{01}^{01}\) in Table~\ref{TableCondExisStab} hold.

At \(E_{01}^{11} = \bigl(\lambda_1^2, X_1^{2*}, S_{21}^{2*}, X_2^{2*}\bigr)\), where
\(X_1^{2*} = \bigl(S_1^{\mathrm{in}} - \lambda_1^2\bigr)/\alpha k_1\) and
\(S_{21}^{2*} = \frac{k_2}{k_1}\bigl(S_1^{\mathrm{in}} - \lambda_1^2\bigr) + S_2^{\mathrm{in}} - \alpha k_3 X_2^{2*}\),
and \(X_2^{2*}\) is a solution of \(f_2(x) = g_2(x)\), the Jacobian matrix \(J_3^1\) is the same as at the steady state \(E_{00}^{1i}\), and \(J_3^3\) is given by
\begin{equation}\label{J33E01-11}
J_3^3 =
\begin{bmatrix}
-D_2 - k_3\mu_2'(S_{21}^{2*})\,X_2^{2*} & -k_3\mu_2(S_{21}^{2*}) \\[4pt]
\mu_2'(S_{21}^{2*})\,X_2^{2*} & \mu_2(S_{21}^{2*}) - \alpha D_2
\end{bmatrix}.
\end{equation}

Similarly to \(E_{00}^{11}\), \(\det(J_3^1)\) is positive and \(\Tr(J_3^1)\) is negative. In addition, the determinant and trace of \(J_3^3\) are given by
\begin{equation}\label{detJ33E01-11}
\det(J_3^3) = D_2 X_2^{2*}\Bigl[g_2'(X_2^{2*}) - f_2'(X_2^{2*})\Bigr],
\end{equation}
\begin{equation}\label{TrJ33E01-11}
\Tr(J_3^3) = -D_2 - X_2^{2*}\Bigl[g_2'(X_2^{2*}) - \frac{f_2'(X_2^{2*})}{\alpha}\Bigr].
\end{equation}
From the expression of \(X_2^{1*}\) in Table~\ref{TabComSS} and the expression of \(x_2^{\mathrm{m}}\) in Lemma~\ref{lemH2}, it follows that \(X_2^{1*} < x_2^{\mathrm{m}}\). Using Lemma~\ref{lemH2}, the equation \(f_2(x) = g_2(x)\) has at least one solution in \((X_2^{1*}, d)\). Hence, \(\mathcal{E}_{01}^{11}\) is LES if and only if \(\det(J_3^3) > 0\), \(\Tr(J_3^3) < 0\), and the stability condition of \(E_{01}\) in Table~\ref{TableStabModelAM2S1S} holds. Therefore, the stability conditions of \(\mathcal{E}_{01}^{11}\) in Table~\ref{TableCondExisStab} hold.

At \(E_{11}^{11} = \bigl(S_1^{2*}, X_1^{2*}, S_{22}^{2*}, X_2^{2*}\bigr)\), the Jacobian matrix \(J_3^1\) is the same as at the steady state \(E_{10}^{10}\), and \(J_3^3\) is given by
\begin{equation}\label{J33E11-11}
J_3^3 =
\begin{bmatrix}
-D_2 - k_3\mu_2'(S_{22}^{2*})\,X_2^{2*} & -k_3\mu_2(S_{22}^{2*}) \\[4pt]
\mu_2'(S_{22}^{2*})\,X_2^{2*} & \mu_2(S_{22}^{2*}) - \alpha D_2
\end{bmatrix}.
\end{equation}

Similarly to \(E_{10}^{10}\), \(\det(J_3^1)\) is positive and \(\Tr(J_3^1)\) is negative. Similarly to \(E_{01}^{11}\), the equilibrium \(E_{11}^{11}\) is LES if and only if \(\det(J_3^3) > 0\) and \(\Tr(J_3^3) < 0\). Since \(E_{11}\) is LES when it exists, it follows that \(\mathcal{E}_{11}^{11}\) is LES if and only if the stability conditions in Table~\ref{TableCondExisStab} hold.
\end{proof}
\begin{proof}[Proposition \ref{propStabE0111E1111}]
By Lemma \ref{lemH2}, if $X_2^{1*} > x_2^{\mathrm{m}}$, the equation $f_2(x)=g_2(x)$ admits $k$ solutions in $(X_2^{1*}, d)$ satisfying \eqref{eq:orderingE0111E1111}, yielding equilibria $\mathcal{E}_{01}^{11,1},\dots,\mathcal{E}_{01}^{11,k}$.
On $(x_2^{\mathrm{m}}, d)$, since $g_2$ is increasing and $f_2$ decreasing, we have for all $X_2^{2*}\in(x_2^{\mathrm{m}}, d)$,
\begin{equation}\label{CondStabE01113}
g_2'(X_2^{2*}) - f_2'(X_2^{2*}) < g_2'(X_2^{2*}) - \tfrac{1}{\alpha}f_2'(X_2^{2*}).
\end{equation}
For $l=1,\dots,k-1$, the stability follows directly:
$l$ odd implies $g_2'-f_2'>0 \Rightarrow \det(J_3^3)>0$, hence $\mathcal{E}_{01}^{11,l}$ is LES iff $\Tr(J_3^3)<0$;  
$l$ even gives $g_2'-f_2'<0 \Rightarrow \det(J_3^3)<0$, hence instability.

For $\mathcal{E}_{01}^{11,k}$, $g_2'-f_2'>0$ together with \eqref{CondStabE01113} implies $g_2'-\tfrac{1}{\alpha}f_2'>0$, so it is LES. The same arguments apply to $\mathcal{E}_{11}^{11,l}$ and $\mathcal{E}_{11}^{11,k}$.
\end{proof}
\section*{Acknowledgments}
The authors thank the support of the Euro-Mediterranean research network \href{https://treasure.hub.inrae.fr/}{Treasure} and the Tunisian Ministry of Higher Education and Scientific Research (Young Researchers' Encouragement Program: 06P1D2024-PEJC).

%
%
%
%

\end{document}